\documentclass[11pt]{article}
\usepackage{amsmath}
\usepackage{amsthm}
\usepackage{amssymb}

\newtheorem{thm}{Theorem}
\newtheorem{theo}{Theorem}[section]

\newtheorem{lemma}[theo]{Lemma}

\begin{document}
\date{}

\title{
Divisible subdivisions
}

\author{Noga Alon
\thanks{Department of Mathematics, Princeton University,
Princeton, NJ 08544, USA and
Schools of Mathematics and
Computer Science, Tel Aviv University, Tel Aviv 6997801,
Israel.
Email: {\tt nogaa@tau.ac.il}.
Research supported in part by
NSF grant DMS-1855464, BSF grant 2018267
and the Simons Foundation.}
\and
Michael Krivelevich
\thanks{
School of Mathematical Sciences,
Tel Aviv University, Tel Aviv 6997801, Israel.
Email: {\tt krivelev@tauex.tau.ac.il}.  Research supported in part
by ISF grant 1261/17 and by USA-Israel BSF grant 2018267.}
}

\maketitle
\begin{abstract}
We prove that for every
graph $H$ of maximum degree at most $3$ and for every
positive integer $q$ there is a finite $f=f(H,q)$ such that every
$K_f$-minor contains a subdivision of $H$ in which every edge
is replaced by a path whose length is divisible by $q$.
For the case of cycles we show that for $f=O(q \log q)$
every $K_f$-minor
contains a cycle of length divisible by $q$, and observe that
this settles a recent problem of Friedman and the second author
about cycles in (weakly) expanding graphs.
\end{abstract}

\section{Introduction}

There are several known results asserting that any graph with a
sufficiently large minimum (or average)
degree contains a cycle of prescribed length
modulo a given parameter. An early result of this form appears in
\cite{Bo77}: for every odd $k$ there exists a $c(k)$ so that every graph
with minimum degree at least $c(k)$ contains a cycle of length
$\ell$ modulo $k$ for every integer $\ell$. A similar result holds
for every even $k$ and every even $\ell$ (but of course not for
even $k$ and odd $\ell$ as shown by dense bipartite graphs.)
Thomassen proved in \cite{Th83} that for non-bipartite $2$-connected
graphs a result as above exists also for even $k$ and odd $\ell$.

For graphs with large chromatic number stronger conclusions hold.
Another result established in \cite{Th83} addresses this case:
For any two positive integers $m$ and $k$ there exists a number
$c(m,k)$ such that the following holds. For every assignment of
two natural numbers $k(e) \leq k$ and $d(e)$ for each edge $e$ of
$K_m$, any graph of chromatic number at least $c(m,k)$ contains
a subdivision of $K_m$ in which each edge $e$ corresponds to a path
of length $d(e)$ modulo $k(e)$.

A more recent result with a similar flavor is proved in \cite{LM20}:
for every $k$ there is a $c(k)$ so that every graph with average degree
at least $c(k)$ contains a subdivision of $K_k$ in which every edge is
subdivided the same number of times.

A common feature of these results and related ones is that they apply
only to graphs with a rather large average degree. In particular,
if the average degree is just a bit above $2$, then none of these
results holds, and indeed there are simple examples showing no such
result can hold if we do not make any additional assumptions.
In this note we prove a result applicable to very sparse graphs.
The sufficient condition we give is based on complete minors; since
there exist minors of arbitrarily large complete graphs with average degree
arbitrarily close to $2$ in every subgraph,
this imposes essentially the weakest
possible condition in terms of graph density.

\begin{thm}
\label{th1}
For every graph $H$ of maximum degree at most $3$, and for every
positive integer $q$ there is a finite $f=f(H,q)$ such that every
$K_f$-minor $G$ contains a subdivision of $H$ in which every edge
is replaced by a path whose length is divisible by $q$.
\end{thm}

\bigskip

\noindent
{\bf Remark:} For every $f$ and $q$ there is a $K_f$-minor $G$ with
maximum degree $3$, in which any path between two vertices
of degree $3$ is of length divisible by $q$; obviously such $G$
does not contain a subdivision of any graph $H$ with maximum degree
$\Delta(H)>3$. Hence, the assumption that $H$ has maximum degree
at most $3$
is needed. Similarly the conclusion about paths of length $0$ modulo $q$
cannot be replaced by any other residue modulo $q$, unlike the results
in \cite{Bo77} or \cite{Th83} mentioned above.

We make essentially no attempt to optimize the value of $f=f(H,q)$, and the
problem of determining its asymptotic behavior remains open.

It is well known that graphs without small separators contain large
complete minors (the line of research establishing these results
started in \cite{AST90} and
culminated with \cite{KR10}). Our main result is thus applicable
to such graphs. Also, non-existence of sublinear separators is
essentially equivalent to weak expansion (see \cite{Kri19} for
a discussion), hence Theorem \ref{th1} can be applied to the class
of (weakly) expanding graphs.

In the very
special case of a cycle of length divisible by $q$ we
get a nearly tight bound, proving the following.
\begin{thm}
\label{th2}
For every positive integer $q$ there is a $g=g(q)=O(q \log q)$ such that
every $K_g$-minor $G$ contains a cycle of length divisible by $q$.
Moreover, if $q$ is a prime then $g(q) < 4q$.
\end{thm}
This, together with the fact that expanders contain large clique minors
(see \cite{Kri19} for a background),
settles a problem raised recently in \cite{FK20}.

For illustrative purposes let us mention that, confirming a conjecture
of Thomassen \cite{Th83}, Gao, Huo, Liu and Ma proved \cite{GHLM21}
that requiring minimum degree $\delta(G)\ge q+1$ guarantees cycles of
all even lengths modulo $q$, including of course cycles of length
divisible by $q$; requiring $\delta(G)\ge q$ is necessary to guarantee
a cycle of length divisible by $q$, at least for odd $q$, as shown by
the complete bipartite graph $K_{q-1,n}$ for $n\ge q-1$. In contrast,
our result allows to argue about existence of cycles of length
divisible by $q$ in much sparser graphs.

The proofs are described in the next two sections. The final section
contains some concluding remarks and open problems.

\section{Subdivisions}
In this section we prove Theorem \ref{th1}. We start
by proving a key lemma. Here and later by the weight
of a path $P$ in an edge-weighted graph $G$ we mean the sum
of the weights of all edges in $P$.
 \begin{lemma}
\label{le1}
There is a function $f_1(k,q)$ satisfying the following condition.
Let $k,q\ge 2$ be integers. Let $T$ be a tree with edges labeled
by the elements of $Z_q$, and let $L$ be a set of specified
leaves of $T$ of cardinality $|L|=f_1(k,q)$. Then there are a
subset $L_0\subset L$ of $k$ leaves and a residue $a\in Z_q$
such that for every three leaves $x_1,x_2, x_3\in L_0$ there
is a  vertex $v\in V(T)$ with all paths from $v$ to $x_i$ in $T$
being disjoint outside of $v$ and having all weight $a$ modulo $q$.
\end{lemma}

\begin{proof}
We set with foresight
$$
f_1=f_1(k,q)=((k-1)q+1)^{(k-1)q^2+1}\,.
$$

Assume that $T,L$ are as given in the lemma. Observe that if the
lemma's conclusion holds for a subtree $T'\subseteq T$
containing $L$ then it holds for $T$. Hence we can assume
that $T$ is a minimal by inclusion tree containing $L$.

For convenience root $T$ at an arbitrary vertex $r\in V(T)$
with $d_T(r)\ge 3$. For a vertex $v\in V(T)$ denote by
$T_v$ the subtree of $T$ rooted at $v$ (with respect to $r$).

Consider first the case where there is a vertex $v\in V(T)$
with $d_T(v)\ge (k-1)q+2$. By the minimality of $T$, for each
child $u$ of $v$ the subtree $T_u$ contains a leaf $\ell(u)\in L$,
the paths from $v$ to all such $\ell(u)$ are disjoint outside of $v$.
By the pigeonhole principle there is a subset $U_0$ of the children
of $v$ in $T$ of cardinality $|U_0|=\lceil\frac{d_T(v)-1}{q}\rceil\ge k$
such that all paths from $v$ to $\ell(u)$, $u\in U_0$, are of
the same weight $a$ modulo $q$. The set $L_0=\{\ell(u): u\in U_0\}$
fulfills then the requirement of the lemma.

Now we treat the complementary case $\Delta(T)\le (k-1)q+1$.
Let $t$ be the maximal number of vertices of degree at least three
on a path from $r$ to a leaf $x\in L$. Since every $x\in L$
is uniquely determined by the sequence of edges leaving the vertices
of degree at least three on the unique path from $r$ to $x$ in $T$,
and the number of such sequences is obviously at most $(\Delta(T))
^t$, we obtain: $|L|\le (\Delta(T))^{t}$, implying $t\ge (k-1)q^2+1$.
Let $P$ be a path from $r$ to a leaf of $T$ with $t$ vertices of
degree at least three in $T$ along it, and let $U_1\subset V(P)$
be these vertices, $|U_1|=t\ge (k-1)q^2+1$. By the pigeonhole
principle, there is a subset $U_2\subset U_1$ of cardinality
$|U_2|=\lceil\frac{|U_1|}{q}\rceil \ge (k-1)q+1$ such that all
subpaths of $P$ from $r$ to $u\in U_2$ have the same weight modulo $q$.
This implies that for every $u_1\ne u_2\in U_2$ the subpath of $P$
between $u_1$ and $u_2$ has weight 0 modulo $q$. By the minimality
of $T$, every vertex $u\in U_2$ contains a leaf $\ell(u)\in L$
in its subtree $T_u$, where all these leaves are distinct.
Applying the
pigeonhole principle again, we derive the existence of a subset
$U_3\subset U_2$ of cardinality $|U_3|=k$ such that every path
between $u\in U_3$ and the corresponding leaf $\ell(u)\in T_u$
has the same weight $a$ modulo $q$. Set $L_0=\{\ell(u): u\in U_3\}$.
We claim that $L_0$ meets the requirement of the lemma.
Indeed, let $u_1,u_2,u_3$ be distinct vertices in $U_3$ ordered in
the order of their appearance along $P$. Then the paths from $u_2$
to $\ell(u_i)$, $i=1,2,3$, are all disjoint outside of $u_2$ and
have total weight $a$ modulo $q$ (the part along $P$ has weight $0$
modulo $q$, and the appended part from $u_i$ to $\ell(u_i)$ has
weight $a(\bmod\ q)$).
\end{proof}

We can now prove Theorem \ref{th1}, making no attempt to
optimize the estimate for $f$. Let $\Gamma$ be a Ramsey
graph in $q$ colors for the $(q-1)$-subdivision of $H$.
Assume that the vertex set of $\Gamma$ is $[N]$,
denote $k=2|E(\Gamma)|$. Let $M=f_1(k,q)$,
with $f_1(k,q)$ from Lemma \ref{le1}. Finally, set:
$$
f=M+(N-1)q+1\,.
$$
Assume $G$ is a minor of $K_f$ with supernodes
$X_1,\ldots,X_{(N-1)q+1},Y_1,\ldots, Y_M$. For a pair of
supernodes $X_i,Y_j$, if $e=(x,y)$ is an edge connecting
$X_i$ to $Y_j$ then split $e$ by a vertex $z$, assign
weights $w((x,z))=0$, $w((z,y))=1$, and append $(x,z)$ to $X_i$.
We assign weight $1$ to all remaining edges of $G$.

Let $T_i$ be a spanning tree of $X_i$ and let $L_i$ be a set
of $M=f_1(q,k)$ distinct leaves of $T_i$, each connected to a
different supernode $Y_j$. Apply Lemma \ref{le1} to $(T_i,L_i)$
to get a subset $L_i'$ of cardinality $|L_i'|=k$ and a
residue $a_i\in Z_q$ with the properties guaranteed by the
lemma. Invoking the pigeonhole principle with respect to
the multiset of residues $\{a_i\}$ we conclude that there
exists a subset $I\subset [(N-1)q+1]$ of cardinality $|I|=N$
with all residues $a_i$, $i\in I$, taking the same value $a$.
By renumbering if necessary we can assume $I=[N]$.

Now we go sequentially over all edges $e=(i_1,i_2)\in E(\Gamma)$
and connect the corresponding supernodes $X_{i_1},X_{i_2}$ as follows:
choose a previously unused supernode $Y_{j_1}$ having a neighbor
in $L_{i_1}'$, choose a distinct and previously unused
supernode $Y_{j_2}$ having a neighbor in $L_{i_2}'$, and then
connect $Y_{j_1}$ and $Y_{j_2}$. Concatenating, we obtain a
path $P_e$ in $G$ from a leaf $x_1\in L_{i_1}'$ to a
leaf $x_2\in L_{i_2}'$, where all these paths are vertex disjoint
for different edges $e=(i_1,i_2)\in E(\Gamma)$.

Define a coloring $c: E(\Gamma)\rightarrow Z_q$ as follows:
for $e=(i_1,i_2)\in E(\Gamma)$, its color $c(e)$ is equal to the
weight modulo $q$ of the path $P_e$ between $X_{i_1}$ and $X_{i_2}$.
By the choice of $\Gamma$, the so obtained coloring $c$
induces a monochromatic copy $H^*$ of the $(q-1)$-subdivision
of $H$, say in color $b\in Z_q$. Let $I_0\subset [N]$ be the subset
of supernodes corresponding to the vertices of $H$ in this subdivision.
By construction, for each edge $f=(i_1,i_2)\in E(H)$, $i_1,i_2\in I_0$,
the graph $G$ contains a path $Q_f$ from $X_{i_1}$ to $X_{i_2}$
passing through a sequence of $q-1$ intermediate supernodes $X_i$,
with the intermediate supernodes being distinct for distinct edges.
Each such path enters and leaves $X_i$ through vertices of $L_i'$;
by the definition of $L_i'$ and the choice of $a$, the entrance
and departure point in $X_i$ can be connected by a path of
length $2a$ modulo $q$. The weight of $Q_f$ between two consecutive
intermediate supernodes $X_i$, and also between $X_{i_1}$ and the
first intermediate supernode, and between the last intermediate
supernode and $X_{i_2}$ is $b$ modulo $q$. Finally, for each
supernode $X_i$, $i\in I_0$, the (at most three) paths $Q_f$
leaving $X_i$ all depart from the vertices of $L_i'$, hence we
can choose a vertex $v_i\in X_i$ connected to the departure
points by disjoint paths of weight $a$ each. Collecting all
weights, we conclude that for each edge $f=(i_1,i_2)\in E(H)$,
the path between $v_{i_1}$ and $v_{i_2}$ in $G$ has total weight:
$$
2a+(q-1)\cdot 2a+q\cdot b= (2a+b)q \equiv 0(\bmod\ q)\,.
$$
We have thus obtained the required subdivision of $H$,
completing the proof. \hfill $\Box$

\section{Cycles}

In this section we prove Theorem \ref{th2}.
In the lemma below and later, a complete digraph is a digraph in
which every pair of vertices is connected by one edge in each of
the two directions.

\begin{lemma}
\label{le3}
Let $q\ge 2$ be an integer, and let $\Gamma=(V,E)$ be a
complete digraph on $\lceil 2q\ln q\rceil$ vertices with
weights $w(e)\in Z_q$ on its edges. Then $\Gamma$ contains a
directed cycle $C$ of total weight divisible by $q$.
\end{lemma}

\begin{proof} The proof borrows its main idea from the argument
in \cite{AL89}. Here though we are in a more favorable situation
dealing with the complete digraph and can thus allow ourselves to
employ a simpler probabilistic tool for the proof --
the union bound (instead of the Local Lemma used in \cite{AL89}).

Let $c: V\to Z_q$ be a random labeling of $V$ by the elements
of $Z_q$. For a vertex $u\in V$ denote by $A_u$ the event
``$u$ has an outneighbor $v$ satisfying $c(v)=c(u)+w(uv)$".
In order to estimate $Pr[\overline{A_u}]$ observe that by
conditioning on the label $c(u)$, the probability that none of
the outneighbors $v$ of $u$ satisfies $c(v)=c(u)+w(uv)$
is $(1-1/q)^{|V|-1}<1/|V|$. Hence by the union bound there is
a choice of $c$ for which all of the events $A_u$ hold.
Fix such a choice, and for every vertex $u\in V$ choose an
outgoing edge $(u,v)$ so that $c(v)=c(u)+w(uv)$.
In the subgraph of $\Gamma$ obtained this way every
outdegree is $1$ and hence it has a directed cycle
$C=(u_1, \ldots,u_\ell,u_1)$. Summing all weights along the edges
of $C$ and denoting $u_{\ell+1}=u_1$ we obtain:
$$
\sum_{i=1}^{\ell} w(u_iu_{i+1})=\sum_{i=1}^{\ell}(c(u_{i+1})-c(u_i))
\equiv 0(\bmod\ q),$$
as required.
\end{proof}

If $q$ is a prime number then the logarithmic term in the above lemma
can be omitted, thus giving an asymptotically optimal order of magnitude.
\begin{lemma}
\label{le31}
Let $q\ge 2$ be a prime, and let $\Gamma=(V,E)$ be a
complete digraph on $2q-1$ vertices with
weights $w(e)\in Z_q$ on its edges. Then $\Gamma$ contains a
directed cycle $C$ of total weight divisible by $q$.
\end{lemma}
\begin{proof}
If there are two vertices $u,v$ so that $w(u,v)=-w(v,u)$ (modulo $q$) then
there is a directed cycle consisting of two edges
satisfying the requirement,
hence we may and will assume that there is no such pair of vertices. In
this case we proceed to prove that for every $k<q$ there
are distinct vertices $x_0,x_1,y_1,x_2,y_2, \ldots x_k,y_k$ in $\Gamma$
and a set $S$ of $k+1$ distinct residues modulo $q$ so that for every
$s \in S$ there is a directed path $P_s$ from $x_0$ to $x_k$ of total weight
$s$ modulo $q$.
Each path $P_s$ consists of $k$ subpaths $p_1,p_2, \ldots ,p_k$,
in this order, where each $p_i$ is either the single edge $x_{i-1}x_i$
or the two-edge path  $x_{i-1}y_ix_i$. This is proved by induction
on $k$. For $k=0$ the required path is the trivial path with no edges
and $S=\{0\}$. Assuming the result holds for $k<q-1$ we prove it for
$k+1$. Let $u,v$ be two vertices of $\Gamma$ that are not in the set
$\{x_0,x_1,x_2, \ldots ,x_{k},y_k\}$. If $w(x_k,u)=w(x_k,v)+w(v,u)$ and
$w(x_k,v)=w(x_k,u)+w(u,v)$ then $w(u,v)=-w(v,u)$ contradicting the
assumption. Thus at least one of these equalities  does not hold; by
renaming $u$ and $v$ if needed we may assume that $w(x_k,u) \neq
w(x_k,v)+w(v,u)$. Define $x_{k+1}=u$ and $y_{k+1}=v$. Let $T$
be the set of the two distinct residues $w(x_k,u)=w(x_k,x_{k+1})$ and
$w(x_k,v)+w(v,u)=w(x_k,y_{k+1})+w(y_{k+1},x_{k+1})$.
Clearly, for every residue $s\in S + T$, there is a path from $x_0$
to $x_{k+1}$ of total weight $s$ in $\Gamma$.
By the Cauchy-Davenport theorem (for the easy special case in which
one of the sets is of size $2$), $|S+T| \geq |S|+|T|-1 =k+1$
establishing the induction step. Taking $k=q-1$,  the result shows that
for every residue class $s$ modulo  $q$ there is a directed simple
path from $x_0$ to $x_{q-1}$ of weight $s$. Choosing $s=-w(x_{q-1},x_0)$
and adding the edge $x_{q-1}x_0$
gives the required cycle.
\end{proof}

We can now prove Theorem \ref{th2}. (In fact, our proof applies to the
general weighted setting.) If $q$ is not a prime define
$N= \lceil 2q \log q \rceil$ and $g=2N$, if it is a prime define
$N=2q-1$ and $g=2N$.
Given a $K_g$-minor on the
$2N$ supernodes $X_i^+$ and $X_i^-$ for $1 \leq i \leq N$, assume
without loss of generality
that the induced subgraph on each supernode is a tree
and that there is exactly one edge connecting each pair of supernodes.
For each $i$ let $b_i$ be the weight of the unique edge, $x_i^-x_i^+$
connecting $X_i^-$ and $X_i^+$ ($1$ if there are no weights).
For each $i \neq j$ let $w'(ij)$ be the total weight
modulo $q$ of the unique
path in the induced tree on the vertices
$X_i^+ \cup X_j^{-}$ from the vertex $x_i^+$ to the vertex
$x_j^-$. By the two lemmas above applied to the auxiliary complete
directed graph on the vertices $\{1,2, \ldots ,N\}$ with the
weights $w(ij)=b_i+w'(ij)$ there is a directed
cycle of total weight $0$ modulo $q$ in this auxiliary digraph.
This gives the
required cycle in the original graph.
\hfill $\Box$

\section{Concluding remarks}
\begin{itemize}
\item
There are several ways to improve the bound for $f(H,q)$ in the proof
of Theorem \ref{th1}. In particular, one can use the constructions
of Ramsey graphs for subdivisions given in
\cite{KRR19,DKN20} as a seed for our proof.  In addition, it is possible
to take $k=\Delta(\Gamma)$, where $\Gamma$ is the Ramsey graph used,
and $M=2|E(\Gamma)|+f_1(q,k)$. Then we first choose $N$
supernodes $X_i$ with the same value of $a_i$ as in the present proof,
fix a bijection between these $N$ nodes and the vertex set
of $\Gamma$, and then for each $i=1,...,N$, look at the leaves of
$X_i$ connected to previously unused supernodes $Y_j$
(altogether we use at most $2|E(\Gamma)|$ such supernodes $Y_j$,
two per each edge of $\Gamma$ throughout the proof).
There are at least $f_1(q,k)$ of them, this would be our set
$L_i$. Next apply to it Lemma \ref{le1} to get a set of cardinality
$k$ (in fact here the degree of vertex $i$ in $\Gamma$ suffices),
and then put aside the supernodes $Y_j$ to which the edges from this
$k$-set of leaves in $L_i$ lead. This gives some improvement,
but as is frequently the case with Ramsey-type results, the bound
obtained is still huge, surely far from being optimal. It may be
interesting to try to determine or to estimate
the asymptotic behavior of the best
possible bound for $f(H,q)$.
\item
An $\alpha$-expander is a graph on $n$ vertices in which
every set $X$ of
at most $n/2$ vertices has at least $\alpha|X|$ neighbors
outside $X$, see \cite{Kri19} for a general discussion.
It is easy to see that any such graph has no sublinear
separators and thus contains a $K_f$-minor for
$f \geq c(\alpha) \sqrt n$ by \cite{KR10}. Our results
thus apply to such graphs and
in particular imply that any such graph contains cycles of length
divisible by $q$
for any $q \leq \tilde{O}(n^{1/2})$.  This settles a question
posed explicitly in \cite{FK20}. See \cite{FK20} for further
results about cycle lengths in $\alpha$-expanders.
\item
There is a substantial amount of research on Ramsey-type problems
for structures labelled by elements of an abelian group. Questions
of this type are called zero-sum problems, see \cite{Ca96}
for a survey of the subject (until the mid 90s).
A typical problem in the subject
is to determine or estimate the smallest number $f$ so that any
complete graph with edges labelled by the elements of $Z_q$
contains a subgraph of a prescribed type in which the total weight
of the edges is $0$ modulo $q$. This problem for complete directed
graphs, where the desired subgraph is a directed cycle, is
addressed in Lemma \ref{le3} and Lemma \ref{le31}. It seems
plausible to believe that the first lemma is not tight and
that the  function $g(q)$ in Theorem \ref{th2}
is linear in $q$ for any integer $q$.
\end{itemize}

\section*{Acknowledgment}
The initial results in this paper were obtained when the
second author visited the Department of Mathematics of Princeton University.
He would like to thank the department for the hospitality.


\begin{thebibliography}{99}
\bibitem{AL89}
N. Alon and N. Linial, {\em Cycles of length $0$
modulo $k$ in directed graphs}, J. Combin. Th. Ser. B 47 (1989), 114--119.
\bibitem{AST90}
N. Alon, P. D. Seymour and R. Thomas, {\em A separator theorem
for non-planar graphs}, J. Amer. Math. Soc. 3 (1990), 801-808.
\bibitem{Bo77}
B. Bollob\'as, {\em Cycles modulo $k$},
Bull. London Math. Soc. 9 (1977), 97--98.
\bibitem{Ca96}
Y. Caro,
{\em Zero-sum problems --- a survey},
Discrete Math. 152 (1996),  93--113.
\bibitem{DKN20}
N. Dragani\'c, M. Krivelevich and R. Nenadov,
{\em The size Ramsey number of short subdivisions},
Random Struct. Alg., to appear.
\bibitem{FK20}
L. Friedman and M. Krivelevich,
{\em Cycle lengths in expanding graphs},
Combinatorica 41 (2021), 53--74.
\bibitem{GHLM21}
J. Gao, Q. Huo, C.-H. Liu and J. Ma,
{\em A unified proof of conjectures on cycle lengths in graphs},
 Int. Math. Res. Not., to appear.
\bibitem{KR10}
K. Kawarabayashi and B.  Reed,
{\em A separator theorem in minor-closed classes},
Proc. 51st Symp. Found. Comp. Sci. (FOCS'10), 2010, 153--162.
\bibitem{KRR19}
Y. Kohayakawa, T. Retter and V. R\"odl,
{\em The size Ramsey number of short subdivisions of bounded degree graphs},
Random Struct. Alg. 54 (2019), 304--339.
\bibitem{Kri19}
M. Krivelevich, {\em Expanders --- how to find them,
and what to find in them},
Surveys in Combinatorics, A. Lo et al., Eds.,
London Math. Soc. Lecture Notes 456, pp. 115--142,
2019.
\bibitem{LM20}
H. Liu and R. Montgomery,
{\em A solution to Erd\H{o}s and Hajnal's odd cycle problem},
arXiv: 2010.15802.
\bibitem{Th83}
C. Thomassen,
{\em Graph decomposition with applications to subdivisions
and path systems modulo $k$},
J. Graph Theory  7 (1983), 261--271.
\end{thebibliography}
\end{document}